\documentclass[11pt]{article}
\usepackage[pdftex]{graphicx}
\usepackage{bm}
\usepackage{amsmath,amsthm,amssymb,url}
\usepackage{multirow,bigdelim}
\usepackage{amscd}

\newtheorem{thm}{Theorem}[section]

\newtheorem{lem}[thm]{Lemma}

\newtheorem{cor}[thm]{Corollary}

  \newcommand{\C}{{\mathcal{C}}}
 
 \newcommand{\X}{{\mathcal{X}}}

\usepackage{tikz}
\usetikzlibrary{arrows}

\title{Sufficient connectivity conditions for rigidity of symmetric frameworks}
\author{Vikt\'oria E. Kaszanitzky\thanks{
Budapest University of Technology and Economics, Magyar tud\'osok krt 2., Budapest, 1117, Hungary and the MTA-ELTE Egerv\'ary Research Group on Combinatorial Optimization, P\'azm\'any P\'eter s\'et\'any 1/C, Budapest, 1117, Hungary
(\texttt{kaszanitzky@cs.bme.hu}).
}
\and
Bernd Schulze \thanks{
Department of Mathematics and Statistics,
Lancaster University,
Lancaster,
LA1 4YF, United Kingdom
(\texttt{b.schulze@lancaster.ac.uk}).}
}

\begin{document}

\maketitle

\begin{abstract}
It is a famous result of Lov\'asz and Yemini (1982) that 6-connected graphs are rigid in the plane. This was recently improved by Jackson and Jord\'an (2009) who showed that 6-mixed connectivity is also sufficient for rigidity.
Here we give sufficient graph connectivity conditions for both `forced symmetric' and `incidentally symmetric' infinitesimal rigidity in the plane.
\end{abstract}

\noindent {\em Key words}: rigidity; symmetric framework, highly connected graphs.


\section{Introduction}\label{sec:intro}

A $d$-dimensional \emph{(bar-joint) framework} is a pair $(\tilde G,p)$, where $\tilde G=(\tilde V, \tilde E)$ is a finite simple graph and $p:\tilde V \to \mathbb{R}^d$ is an injective map. A $d$-dimensional framework is called \emph{rigid} if, loosely speaking, the vertices cannot be moved continuously in $\mathbb{R}^d$ to obtain another non-congruent framework while keeping the lengths of all edges fixed. A classical approach to study the rigidity of frameworks is to differentiate the length constraints on the edges, which gives rise to the linear theory of infinitesimal rigidity \cite{W1}.
An \emph{infinitesimal motion} of 
$(\tilde G,p)$ 
is a function $u: \tilde V\to \mathbb{R}^{d}$ such that
\begin{equation}
\label{infinmotioneq}
\langle p_i-p_j, u_i-u_j\rangle =0 \quad\textrm{ for all } \{i,j\} \in \tilde E\textrm{,}
\end{equation}
where $p_i=p(i)$ and $u_i=u(i)$ for each $i$. An infinitesimal motion $u$ of $(\tilde G,p)$ is a \emph{trivial infinitesimal motion}
if there exists a skew-symmetric matrix $S$
and a vector $t$ such that $u_i=S p_i+t$ for all $i\in \tilde V$.
$(\tilde G,p)$ is \emph{infinitesimally rigid} if every infinitesimal motion of $(\tilde G,p)$ is trivial, and \emph{infinitesimally flexible} otherwise.

A framework $(\tilde G,p)$ is called \emph{generic} if the coordinates of the image of $p$ are algebraically independent over $\mathbb{Q}$. It is well known that  for generic frameworks, rigidity is equivalent to infinitesimal rigidity \cite{asiroth}. Laman's landmark result from 1970 gives a combinatorial characterisation of generic  rigid frameworks in $\mathbb{R}^2$ \cite{laman}. An alternative proof of this result based on  matroid theory  was given  in \cite{LY}. In that same paper  Lov\'asz and Yemini also established sufficient graph connectivity conditions for the  rigidity of generic frameworks in $\mathbb{R}^2$. Their result was recently improved by Jackson and Jord\'an in \cite{jj6}. Analogous results for higher dimensions have not yet been found and remain key open problems in the field.

Rigidity theory has applications in various areas of science and technology, ranging from mechanical and structural engineering through robotics and CAD programming to materials science and biochemistry. Since many structures in these areas of application exhibit non-trivial symmetries, the study of how symmetry impacts the rigidity and flexibility of frameworks has become a highly active research area in recent years (see \cite{ithesis, jkt, MT2, ST, BSWWorbit, DCGsym} for example).

There are two basic approaches to this problem. First, one may ask whether a framework is `forced symmetric rigid', that is it cannot be deformed without breaking the original symmetry of the structure. 
Combinatorial characterisations of the graphs whose generic realisations (modulo the given symmetry constraints) are forced symmetric rigid have been established for all symmetry groups in the plane, except for dihedral groups of order $2k$, where $k$ is even \cite{jkt,MT2}. More generally,
one may ask if a symmetric framework is infinitesimally rigid, i.e., whether it does not have \emph{any} non-trivial deformations. This problem of analysing the rigidity of `incidentally symmetric' frameworks is  more complex. However, combinatorial characterisations for symmetry-generic infinitesimal rigidity have recently been established for a number of cyclic groups in the plane \cite{ithesis,ST}.

In this paper we extend the sufficient graph connectivity conditions for generic rigidity established in \cite{jj6,LY} to both forced symmetric and incidentally symmetric frameworks. In Section~\ref{sec:symfw}, we first reprise some basic symmetry terminology and summarize the combinatorial characterisations of forced symmetric and incidentally symmetric rigid frameworks for all the groups in the plane for which such a characterisation is known. Moreover, we provide formulas for the rank functions of the corresponding linear count matroids defined on the edge sets of the underlying quotient gain graphs. In Section~\ref{sec:coveringsufficient} we then present our main results, which are sufficient graph connectivity conditions for symmetric frameworks to be forced symmetric rigid or infinitesimally rigid for all relevant symmetry groups in the plane. Moreover, we provide examples that show that our conditions are best possible. We prove these results in Sections~\ref{sec:gainmixed} and \ref{sec:suffforced} by relating the connectivity conditions on the symmetric graphs given in Section~\ref{sec:coveringsufficient} with connectivity conditions on the corresponding quotient gain graphs and by using the rank functions given in Section~\ref{sec:symfw}.

\section{Rigidity of symmetric frameworks}\label{sec:symfw}

\subsection{Symmetric graphs}\label{subsec:symgraphs}

Let $\tilde{G}=(\tilde{V},\tilde{E})$ be a finite simple graph. An {\em action} of a group $\Gamma$ on $\tilde{G}$ is a group homomorphism $\theta:\Gamma \rightarrow {\rm Aut}(\tilde{G})$, where ${\rm Aut}(\tilde{G})$ denotes the automorphism group of $\tilde{G}$.
An action $\theta$ is called {\em free} on $\tilde{V}$ (resp., $\tilde{E}$)
if  $\theta(\gamma)(i)\neq i$ for every $i\in \tilde{V}$ (resp., $\theta(\gamma)(e)\neq e$ for every $e\in \tilde{E}$) and
every non-identity $\gamma\in \Gamma$.
We say that a graph $\tilde{G}$ is {\em $\Gamma$-symmetric} (with respect to $\theta$)
if $\Gamma$ acts on $\tilde{G}$ by $\theta$.
In the following we will frequently
omit to specify the action $\theta$ if it is clear from the context.
We then denote $\theta(\gamma)(i)$ by  $\gamma i$. For simplicity, we will assume throughout this paper that $\theta$ acts freely on $\tilde{V}$.

 For a $\Gamma$-symmetric graph $\tilde{G}=(\tilde{V},\tilde{E})$, the \emph{quotient $\Gamma$-gain graph} of $\tilde{G}$ is the pair $(G,\psi)$, where  $G=(V,E)$ is the quotient graph of $\tilde{G}$, together  with an orientation on the edges, and  $\psi:E\to \Gamma$ is an edge labelling  defined as follows.
Each edge orbit $\Gamma e$ connecting $\Gamma i$ and $\Gamma j$ in $\tilde{G}/\Gamma$ can be written as $\{\{\gamma i,\gamma \circ\alpha j\}\mid \gamma\in \Gamma \}$ for a unique $\alpha\in \Gamma$. For each $\Gamma e$, orient $\Gamma e$ from $\Gamma i$ to $\Gamma j$ in $\tilde{G}/\Gamma$ and assign to it the gain $\alpha$. Then $E$ is the resulting set of oriented edges, and $\psi$ is the corresponding gain assignment. See Figure~\ref{fig:gain}(b) for an example of a quotient $\Gamma$-gain graph.

Note that $(G,\psi)$ is unique up to choices of representative vertices. Moreover, the orientation is only used as a reference orientation and may be changed, provided that we also modify $\psi$ so that if  $e$ is an edge in one direction, and $e^{-1}$ is the same edge in the opposite direction, then $\psi(e^{-1})=\psi(e)^{-1}$.

Let $\tilde{G}$ be a finite simple $\Gamma$-symmetric graph and let $(G,\psi)$ be its quotient $\Gamma$-gain graph. Then $\tilde{G}$ is called the \emph{covering graph} of $(G,\psi)$. Furthermore, the map $c:\tilde{G}\rightarrow G$ which maps every element of a vertex orbit of $\tilde{G}$ to its representative vertex in $G$ and every element of an edge orbit of $\tilde{G}$ to its representative edge in $G$ is called a {\em  covering map}.

\subsection{Symmetric frameworks}\label{subsec:symfw}

Let $\tilde{G}$ be a $\Gamma$-symmetric graph (with respect to $\theta:\Gamma\to \textrm{Aut}(\tilde{G})$), and let $\Gamma$ act on $\mathbb{R}^d$ via the homomorphism $\tau:\Gamma\rightarrow O(\mathbb{R}^d)$.
A framework $(\tilde{G},p)$ is called {\em $\Gamma$-symmetric} (with respect to $\theta$ and $\tau$) if
\begin{equation}
\label{eq:symmetric_func}
\tau(\gamma) (p(i))=p(\theta(\gamma) i) \qquad \text{for all } \gamma\in \Gamma \text{ and all } i\in \tilde{V}.
\end{equation}

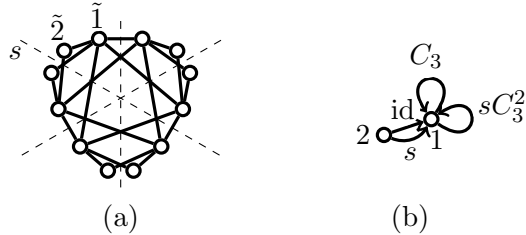
\begin{figure}[htp]
\begin{center}
\begin{tikzpicture}[very thick,scale=0.7]
\tikzstyle{every node}=[circle, draw=black, fill=white, inner sep=0pt, minimum width=5pt];

\draw[thin,dashed] (0,-2.3)--(0,1.6);
\draw[thin,dashed] (30:2.2cm)--(210:2.2cm);
\draw[thin,dashed] (150:2.2cm)--(330:2.2cm);

\node (p1) at (70:1.2cm) {};
\node (p11) at (190:1.2cm) {};
\node (p111) at (310:1.2cm) {};

\node (p2) at (110:1.2cm) {};
\node (p22) at (230:1.2cm) {};
\node (p222) at (350:1.2cm) {};

\node (q1) at (40:1.4cm) {};
\node (q11) at (160:1.4cm) {};
\node (q111) at (280:1.4cm) {};

\node (q2) at (20:1.4cm) {};
\node (q22) at (140:1.4cm) {};
\node (q222) at (260:1.4cm) {};

\draw(p1)--(p2);
\draw(p2)--(q22);
\draw(p2)--(q11);
\draw(p2)--(p22);
\draw(p2)--(p222);

\draw(p11)--(p22);
\draw(p22)--(q222);
\draw(p22)--(q111);
\draw(p22)--(p222);

\draw(p111)--(p222);
\draw(p222)--(q2);
\draw(p222)--(q1);

\draw(p11)--(q22);
\draw(p11)--(q11);
\draw(p111)--(q222);
\draw(p111)--(q111);
\draw(p1)--(q2);
\draw(p1)--(q1);

\draw(p1)--(p111);
\draw(p11)--(p1);
\draw(p111)--(p11);

\node [draw=white, fill=white] (a) at (-0.45,1.58)  {$\tilde{1}$};
\node [draw=white, fill=white] (a) at (-1.2,1.35)  {$\tilde 2$};

\node [draw=white, fill=white] (a) at (-2,0.86)  {$s$};
\node [draw=white, fill=white] (a) at (0,-2.3)  {(a)};


\node [draw=white, fill=white] (a) at (4.6,-0.7)  {$2$};
\node [draw=white, fill=white] (a) at (6,-0.79)  {$1$};

\node (a) at (5,-0.7) {};
\node (b) at (5.9,-0.4) {};
\draw[->](a)--(b);
\path (a) edge [->,bend right=42] (b);

\draw[->] (b) to [out=55,in=125,looseness=15] (b);
\draw[->] (b) to [out=325,in=35,looseness=15] (b);

\node [draw=white, fill=white] (a) at (5.36,-0.2)  {$\textrm{id}$};
\node [draw=white, fill=white] (a) at (5.5,-1.05)  {$s$};
\node [draw=white, fill=white] (a) at (5.8,0.8)  {$C_3$};
\node [draw=white, fill=white] (a) at (7.25,-0.15)  {$sC_3^2$};

\node [draw=white, fill=white] (a) at (5.5,-2.3)  {(b)};
\end{tikzpicture}
\end{center}
\vspace{-0.4cm}
\caption{A framework with dihedral symmetry $\tau(\Gamma)=\mathcal{C}_{3v}=\langle s,C_3\rangle$  (a) and its corresponding quotient $\Gamma$-gain graph (b).}
\label{fig:gain}\end{figure}

Let $G=(V,E)$ be the quotient $\Gamma$-gain graph of $\tilde{G}$ with the covering map $c:\tilde{G}\rightarrow G$.
It is convenient to fix a representative vertex $i$ of each vertex orbit $\Gamma i=\{\gamma  i\colon \gamma\in \Gamma\}$,
and define the {\em quotient} of $p$ to be $p':V\rightarrow \mathbb{R}^d$,
so that there is a one-to-one correspondence between $p$ and $p'$ given by
$p(i)=p'(c(i))$ for each representative vertex $i$.

For the group $\tau(\Gamma)$, let $\mathbb{Q}_{\Gamma}$ be the field
 generated by $\mathbb{Q}$ and the entries of the matrices in $\tau(\Gamma)$.
We say that $p$  is {\em  $\Gamma$-generic}
if the set of coordinates of the image of $p'$ is algebraically independent over $\mathbb{Q}_{\Gamma}$.
Note that this definition does not depend on the choice of representative vertices. A $\Gamma$-symmetric framework $(\tilde{G},p)$ is called
\emph{$\Gamma$-generic} if $p$ is $\Gamma$-generic.

Throughout this paper, we will use the Schoenflies notation to describe the symmetries of frameworks. Note that in dimension 2, $\tau(\Gamma)$ can only be a reflection group of order $2$ (denoted by $\mathcal{C}_s$), a rotational group of order $k$ generated by a rotation $C_k$ about the origin by $2 \pi/k$, $ k\in\mathbb{N}$ (denoted by $\mathcal{C}_k$), or a dihedral group of order $2k$ generated by a reflection and a rotation $C_k$ (denoted  by $\mathcal{C}_{kv}$).

\subsection{Forced symmetric rigidity}\label{subsec:forced}

An infinitesimal motion $u$ of a $\Gamma$-symmetric framework $(\tilde G,p)$ is called \emph{$\Gamma$-symmetric} (with respect to $\theta$ and $\tau$) if the velocity vectors exhibit the same symmetry as $(\tilde G,p)$, that is, if $\tau(\gamma)u_i=u_{\gamma i}$ for all $\gamma\in \Gamma$ and all $i\in \tilde V.$
Moreover, we say that $(\tilde G,p)$ is \emph{$\Gamma$-symmetric infinitesimally rigid} if every $\Gamma$-symmetric infinitesimal motion is trivial.

A key motivation for studying $\Gamma$-symmetric infinitesimal rigidity is that for $\Gamma$-generic frameworks, there exists a non-trivial $\Gamma$-symmetric infinitesimal motion if and only if there exists a non-trivial symmetry-preserving \emph{continuous} motion \cite{BSWWorbit,DCGsym}.

For a $d$-dimensional $\Gamma$-symmetric framework $(\tilde G,p)$, a symmetric analog of the rigidity matrix $R(\tilde G,p)$ \cite{W1}, known as the \emph{orbit rigidity matrix} was introduced in \cite{BSWWorbit}. This matrix is  of size $|E| \times d|V|$ and completely describes the $\Gamma$-symmetric infinitesimal rigidity properties of $(\tilde G,p)$. In particular, its kernel is isomorphic to the space of  $\Gamma$-symmetric infinitesimal motions of $(\tilde G,p)$. 
We define the {\em  rigidity matroid} of $(G,\psi)$, $\mathcal{R}_{\tau}(G,\psi)$, to be the row matroid of the orbit rigidity matrix of a $\Gamma$-generic realisation of $G$ (with respect to  $\theta$ and $\tau$). The bases of this matroid have been characterised for $\mathcal{C}_s$, $\mathcal{C}_k$, $k\in \mathbb{N}$, and $\mathcal{C}_{(2k+1)v}$, $k\in \mathbb{N}$, in \cite{jkt,MT2}. (For the groups $\mathcal{C}_{(2k)v}$, however, this problem  is still open \cite{jkt}.) To state these results, we need the following definitions.

Let $(G,\psi)$ be a quotient $\Gamma$-gain graph. The \emph{gain} $\psi(W)$ of a closed  walk $W$ in $(G,\psi)$ of the form $v_1,e_1,v_2,e_2,v_3,\dots,v_l,e_l,v_1$ is defined as $\Pi_{i=1}^l \psi(e_i)^{{\rm sign}(e_i)}$, where ${\rm sign}(e_i)=1$ if $e_i$ is directed from $v_i$ to $v_{i+1}$, and ${\rm sign}(e_i)=-1$  otherwise.  For $E'\subseteq E$ and $i\in V(E')$, we define the \emph{subgroup of $\Gamma$ induced by $E'$} as $\langle E' \rangle_{\psi,i}=\{\psi(W): \, W\in \mathcal{W}(E',i)\}$, where  $\mathcal{W}(E',i)$ is the set of closed walks starting at $i$ using only edges of $E'$.
A connected subset $E'\subseteq E$ is  called \emph{balanced} if $\langle E' \rangle_{\psi,i}=\{\textrm{id}\}$ for some $i\in V(E')$ (or equivalently for all $i\in V(E')$). Further, a connected subset $E'\subseteq E$ is  called \emph{cyclic} if $\langle E' \rangle_{\psi,i}$ is a cyclic subgroup of $\Gamma$ for some $i\in V(E')$ (or equivalently for all $i\in V(E')$). In what follows only certain properties of $\langle E' \rangle_{\psi,i}$ are important rather than the group itself. These properties are invariant under conjugation and because $\langle E' \rangle_{\psi,i}$ is a conjugate of $\langle E' \rangle_{\psi,j}$ for every $j\in V(E')$ we will drop the index referring to the starting point of the walks.
A (possibly disconnected) subset $E'\subseteq E$ is called \emph{balanced} (\emph{cyclic}, resp.) if each of its connected components is balanced (cyclic). A subset $E'\subseteq E$ which is not balanced is called \emph{unbalanced}.

Given a quotient $\Gamma$-gain graph $(G,\psi)$, let \(\rho\) be the function on  \(E\) defined  by \(\rho(X)=2|V(X)|-3+\beta(X)\) for \(X\subseteq E\) where
\[\beta(X)= \begin{cases}
	 0 \hbox{ if X is balanced;}\\
	 2 \hbox{ if X is unbalanced and cyclic;}\\
	 3 \hbox{ otherwise.}
\end{cases}\]

\begin{thm} \label{thm:symlaman} \cite{jkt} Let $(\tilde G,p)$ be a $\Gamma$-generic framework (with respect to $\theta$ and $\tau$) such that $\tau(\Gamma)$ is  $\mathcal{C}_s$, $\mathcal{C}_k$ or $\mathcal{C}_{(2k+1)v}$. Further, let $(G,\psi)$ be the quotient $\Gamma$-gain graph of $\tilde G$ with \(G=(V,E)\), and let $E'\subseteq E$.
Then $E'$ is independent in $\mathcal{R}_{\tau}(G,\psi)$ if and only if $\rho(F)\geq
|F|$ for all $\emptyset\neq F\subseteq E'$. Further, $(\tilde G,p)$ is $\Gamma$-symmetric infinitesimally rigid if and only if $(G,\psi)$ contains a spanning independent set of $2|V|-1$ or $2|V|$ edges, depending on wether $\Gamma$ is a non-trivial cyclic or dihedral group.
\end{thm}

The rank function of  $\mathcal{R}_{\tau}(G,\psi)$ is given by the following formula.

\begin{thm} \label{thm:symrank} \cite{jkt} Let $(\tilde G,p)$ be a $\Gamma$-symmetric framework (with respect to $\theta$ and $\tau$) such that $\tau(\Gamma)$ is  $\mathcal{C}_s$, $\mathcal{C}_k$ or $\mathcal{C}_{(2k+1)v}$.
Further let \((G,\psi)\) be the quotient \(\Gamma\)-gain graph of $\tilde G$ with \(G=(V,E)\). The rank of a set \(E'\subseteq E\) in $\mathcal{R}_{\tau}(G,\psi)$ is equal to
\[\min\left\{\sum_{i=1}^{s}\rho(E_i):\{E_1,\dots,E_s\}\hbox{ is a partition of } E'\right\}.\]
\end{thm}

\subsection{Infinitesimal rigidity of symmetric frameworks}\label{subsec:inc}

It is well known that the rigidity matrix of a $\Gamma$-symmetric framework can be transformed into a block-diagonalised form, with each block corresponding to an irreducible representation of $\Gamma$ \cite{ST,DCGsym}. The problem of analysing the infinitesimal rigidity  of a $\Gamma$-symmetric framework can therefore be broken up into independent subproblems, one for each irreducible representation. Using this approach, combinatorial characterisations of $\Gamma$-generic infinitesimally rigid frameworks have been established for a selection of cyclic groups in \cite{ithesis,ST}. We need the following definitions.

Let $\Gamma$ be the group $\mathbb{Z}_k=\{0,1,\ldots, k-1\}$ and for $t=0,1,\ldots, k-1$, let $\iota_t:\Gamma\rightarrow \mathbb{C}\setminus\{0\}$ be the irreducible representation of $\Gamma$ defined by $\iota_t(j)= \omega^{tj}$,
where  $\omega$ denotes the root of unity $e^{\frac{2\pi i}{k}}$. For a $\Gamma$-symmetric framework $(\tilde G,p)$,
an infinitesimal motion $u:\tilde V\rightarrow \mathbb{R}^d$ of $(\tilde G,p)$ is called \emph{$\iota_t$-symmetric} if it satisfies
\begin{equation*}
\tau(\gamma)u_i= \omega^{t \gamma} u_{\gamma i}\qquad \text{ for all } \gamma\in \Gamma \text{ and all } i\in \tilde V.
\end{equation*}
A $\Gamma$-symmetric framework $(\tilde G,p)$ is called \emph{$\iota_t$-symmetric infinitesimally rigid} if  every $\iota_t$-symmetric infinitesimal motion of $(\tilde G,p)$ is trivial.

\begin{thm} \label{thm:charcsc2} \cite{ST} A $\Gamma$-generic framework $(\tilde G,p)$ (with respect to $\theta$ and $\tau$)  is infinitesimally rigid if and only if it is $\iota_t$-symmetric infinitesimally rigid for every irreducible representation $\iota_t$ of $\Gamma$.
\end{thm}

Note that $\iota_0$ is the trivial irreducible representation of $\Gamma$ which assigns $1$ to each $\gamma\in \Gamma$. Therefore, a framework is $\iota_0$-symmetric infinitesimally rigid
if and only if it is $\Gamma$-symmetric infinitesimally rigid.

For a $\Gamma$-symmetric framework $(\tilde G,p)$, an orbit rigidity matrix $O_t(\tilde G,p)$ was introduced in  \cite{ST} for each irreducible representation $\iota_t$ of $\Gamma$. (For $ t=0$, the matrix $O_0(\tilde G,p)$ is the orbit rigidity matrix discussed in Section~\ref{subsec:forced}). Analogous to the case $t=0$, the matrix $O_t(\tilde G,p)$  completely describes the $\iota_t$-symmetric infinitesimal rigidity properties of $(\tilde G,p)$ for each $t$. In other words, $O_t(\tilde G,p)$ is equivalent to the block matrix corresponding to $\iota_t$ in the block-diagonalised rigidity matrix of $(\tilde G,p)$.
We define the \emph{$\iota_t$-symmetric rigidity matroid} of $(G,\psi)$, $\mathcal{R}_\tau^{t}(G,\psi)$, to be the row matroid of $O_t(\tilde G,p)$ for a $\Gamma$-generic realisation of $G$ (with respect to $\theta$ and $\tau$).

As shown in \cite{ST},  a $\Gamma$-symmetric framework in the plane with $\Gamma=\mathbb{Z}_2$ has a $1$-dimensional space of trivial $\iota_0$-symmetric infinitesimal motions, and a $2$-dimensional space of trivial $\iota_1$-symmetric infinitesimal motions. Moreover, a $\Gamma$-symmetric framework in the plane with $\Gamma=\mathbb{Z}_k$, where $k\geq 3$, has a $1$-dimensional space of trivial $\iota_t$-symmetric infinitesimal motions for $t=0,1$ and $k-1$, and no trivial $\iota_t$-symmetric infinitesimal motion for $t\neq 0,1,k-1$. This gives rise to the following characterisations of the $\iota_t$-symmetric rigidity matroids.

Given a quotient $\Gamma$-gain graph $(G,\psi)$, where $\Gamma=\mathbb{Z}_2$, let \(\mu\) be the function on  \(E\) defined  by \(\mu(X)=2|V(X)|-3+\beta_1(X)\) for \(X\subseteq E\) where
\[\beta_1(X)= \begin{cases}
	 0 \hbox{ if X is balanced;}\\
	 1 \hbox{ otherwise.}
\end{cases}\]

\begin{thm} \label{thm:antisymlaman} \cite{ST} Let $(\tilde G,p)$ be a $\Gamma$-generic framework (with respect to $\theta$ and $\tau$) such that $\tau(\Gamma)$ is  $\mathcal{C}_s$ or $\mathcal{C}_2$. Further, let $(G,\psi)$ be the quotient $\Gamma$-gain graph of $\tilde G$ with \(G=(V,E)\), and let $E'\subseteq E$.
Then $E'$ is independent in $\mathcal{R}_{\tau}^1(G,\psi)$ if and only if $\mu(F)\geq
|F|$ for all $\emptyset\neq F\subseteq E'$. Further, $(\tilde G,p)$ is $\iota_1$-symmetric infinitesimally rigid if and only if $(G,\psi)$ contains a spanning independent set of $2|V|-2$ edges.
\end{thm}

By Theorem~\ref{thm:charcsc2}, Theorems~\ref{thm:symlaman} and \ref{thm:antisymlaman} provide a combinatorial characterisation of $\Gamma$-generic infinitesimally rigid frameworks for $\Gamma=\mathbb{Z}_2$.

Note that the matroid $\mathcal{R}_\tau^{1}(G,\psi)$ is the Dilworth truncation of the union of the graphic matroid and the frame matroid (or bias matroid) of $(G,\psi)$. We have the following formula for the rank function of $\mathcal{R}_\tau^{1}(G,\psi)$.

\begin{thm} \cite{it} \label{lem:2,3,2} Let $(\tilde G,p)$ be a $\Gamma$-symmetric framework (with respect to $\theta$ and $\tau$) such that $\tau(\Gamma)$ is  $\mathcal{C}_s$ or $\mathcal{C}_2$.
Further let \((G,\psi)\) be the quotient \(\Gamma\)-gain graph of $\tilde G$ with \(G=(V,E)\).
The rank of a set \(E'\subseteq E\) in $\mathcal{R}_\tau^{1}(G,\psi)$ is equal to
$$\min \left\{ \sum_{i=1}^s \mu(E_i)\colon \{E_1,\dots, E_s\}\ \hbox{ is a partition of } E' \right\}.$$
\end{thm}

It was shown in \cite{ST} that for the three-fold rotational group $\tau(\Gamma)=\mathcal{C}_3$, a $\Gamma$-generic framework is  $\Gamma$-symmetric infinitesimally rigid if and only if it is infinitesimally rigid.

The only other groups for which combinatorial characterisations of $\Gamma$-generic infinitesimally rigid frameworks have been found are the cyclic groups of order $k$, where $k<1000$ is odd \cite{ithesis}. To state the result, we need the following definition.

A \emph{split} of a vertex \(v\) of a quotient $\Gamma$-gain graph $(G,\psi)$ is defined as follows. We can assume that every edge incident with \(v\) is directed from \(v\). Take a 2-partition \({E_1,E_2}\) of non-loop edges incident with \(v\). Replace \(v\) with a pair of vertices \(v_1,v_2\). Replace every edge \(vu\in E_i\) with the edge \(v_iu\) of the same label for \(i=1,2\). Then replace every  (necessarily unbalanced) loop  incident with \(v\) with an arc \(v_1v_2\) of the same label. We say that a connected set \(F\) is \emph{near-balanced} if it is not balanced and there is a split of \((G,\psi)\) in which \(F\) results in a balanced set.

Given a quotient $\Gamma$-gain graph $(G,\psi)$, where $\Gamma=\mathbb{Z}_k$, and an irreducible representation $\iota_t$ of $\Gamma$, let \(\nu_t\) be the function on  \(E\) defined  by \(\nu_t(X)=2|V(X)|-3+\alpha_t(X)\) for \(X\subseteq E\) where
\[\alpha_t(X)= \begin{cases}
	 0  \quad \hbox{  if X is balanced;}\\
	 2 \quad \hbox{ if X is near-balanced or satisfies $\langle X \rangle_{\psi,i}\simeq \mathbb{Z}_l$  for some}\\ \quad \quad \hbox { $2\leq l\leq k;\, t\equiv 0 \textrm{ or }1 \textrm{ or } -1 \pmod{l}$;}\\
 3 \quad\hbox{ otherwise.}
\end{cases}\]

\begin{thm} \label{thm:antisymlamancyc} \cite{it} Let $(\tilde G,p)$ be a $\Gamma$-generic framework (with respect to $\theta$ and $\tau$) such that $\tau(\Gamma)$ is  $\mathcal{C}_k$, \(5\leq k<1000\) odd. Further, let $(G,\psi)$ be the quotient $\Gamma$-gain graph of $\tilde G$ with \(G=(V,E)\), and let $E'\subseteq E$.
\begin{enumerate}
  \item[(i)] Then $E'$ is independent in $\mathcal{R}_{\tau}^t(G,\psi)$ if and only if $\nu_t(F)\geq|F|$ for all $\emptyset\neq F\subseteq E'$.
  \item[(ii)] Further, $(\tilde G,p)$ is $\iota_t$-symmetric infinitesimally rigid if and only if $(G,\psi)$ contains a spanning independent set of $2|V|-1$ edges if $t= 0,1,k-1$, and a spanning independent set of $2|V|$ edges otherwise.
  \item[(iii)] The rank of \(E'\) in the matroid $\mathcal{R}_\tau^{t}(G,\psi)$ is equal to
$$\min\left\{\sum_{i=1}^{s}\nu_t(E_i)\colon \{E_1,\dots, E_s\}\ \hbox{ is a partition of } E' \right\}.$$
\end{enumerate}
\end{thm}

\section{Sufficient connectivity conditions for rigidity in term of covering graphs}\label{sec:coveringsufficient}
In this section we give sufficient conditions for symmetric rigidity (for both forced symmetric and incidentally symmetric frameworks) in which the sufficiency is ensured by the high (mixed-) connectivity of the covering graph. Moreover, we give examples that show that all of our results are sharp.

Following \cite{jj6}, we say that a graph \(\tilde G=(\tilde V, \tilde E)\)  is \emph{$n$-mixed-connected} if \(\tilde G-W-F\) is connected for all sets \(W\subseteq \tilde V\) and \(F\subseteq \tilde E\) which satisfy \(2|W|+|F|\leq n-1\). Note that for $n=6$, for example, \(\tilde G\) is 6-mixed-connected if and only if \(\tilde G\) is 6-edge-connected, \(\tilde G-v\) is 4-edge-connected and \(\tilde G-v-u\) is 2-edge-connected for every \(v,u\in \tilde V\).

It turns out that in some cases our sufficient conditions cannot be given purely in terms of the mixed-connectivity of the covering graph, but they also require a certain type of connectivity for the corresponding quotient $\Gamma$-gain graph. It is therefore natural to also derive sufficient conditions for symmetric rigidity purely in terms of the connectivity of the quotient $\Gamma$-gain graphs. This is done in Section \ref{sec:suffforced}, based on an appropriate notion of mixed-connectivity for quotient $\Gamma$-gain graphs introduced in Section~\ref{sec:gainmixed}.  In fact, all the sufficient conditions for symmetric rigidity given in this section  are corollaries of the results in Sections \ref{sec:gainmixed} and \ref{sec:suffforced}.

\subsection{Forced symmetric rigidity}
In this subsection our main result is the following:

\begin{thm}\label{thm:forcedmain}
Let $(\tilde G,p)$ be a $\Gamma$-generic framework (with respect to $\theta$ and $\tau$) such that $\tau(\Gamma)$ is  $\mathcal{C}_s$, $\mathcal{C}_k$ or $\mathcal{C}_{(2k+1)v}$,
and let $(G,\psi)$ be the quotient $\Gamma$-gain graph of $\tilde G$. Suppose \(\tilde G\) is 6-mixed-connected. If \(|\Gamma|\geq6\) then suppose further that \((G,\psi)\) is
2-edge-connected. Then \((\tilde G,p)\) is $\Gamma$-symmetric infinitesimally rigid.
\end{thm}

Theorem \ref{thm:forcedmain} is an immediate corollary of  Lemma \ref{lem:gainconn}(a) and Theorem \ref{thm:forcedgainmain} which we will prove in the following sections.

The  examples in Figure~\ref{fig:2conn5conn} show that Theorem~\ref{thm:forcedmain} is best possible.

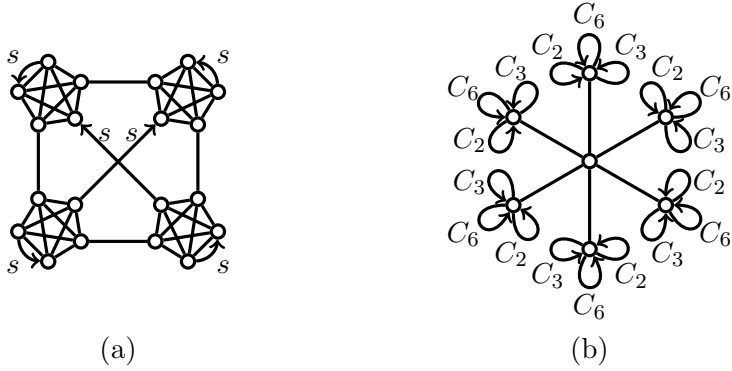
\begin{figure}[htp]
\begin{center}
\begin{tikzpicture}[very thick,scale=0.9]
\tikzstyle{every node}=[circle, draw=black, fill=white, inner sep=0pt, minimum width=5pt];

\node (p0) at (45:0.9cm) {};
\node (p1) at (135:0.9cm) {};
\node (p2) at (225:0.9cm) {};
\node (p3) at (315:0.9cm) {};

\node (p0r) at (25:1.3cm) {};
\node (p1r) at (115:1.3cm) {};
\node (p2r) at (205:1.3cm) {};
\node (p3r) at (295:1.3cm) {};

\node (p0l) at (65:1.3cm) {};
\node (p1l) at (155:1.3cm) {};
\node (p2l) at (245:1.3cm) {};
\node (p3l) at (335:1.3cm) {};

\node (p0rr) at (35:1.8cm) {};
\node (p1rr) at (125:1.8cm) {};
\node (p2rr) at (215:1.8cm) {};
\node (p3rr) at (305:1.8cm) {};

\node (p0ll) at (55:1.8cm) {};
\node (p1ll) at (145:1.8cm) {};
\node (p2ll) at (235:1.8cm) {};
\node (p3ll) at (325:1.8cm) {};

\draw(p0)--(p0r);
\draw(p0r)--(p0rr);
\draw(p0rr)--(p0ll);
\draw(p0ll)--(p0l);
\draw(p0l)--(p0);
\draw(p0)--(p0rr);
\draw(p0)--(p0ll);
\draw(p0r)--(p0ll);
\draw(p0r)--(p0l);
\draw(p0rr)--(p0l);

\draw(p1)--(p1r);
\draw(p1r)--(p1rr);
\draw(p1rr)--(p1ll);
\draw(p1ll)--(p1l);
\draw(p1l)--(p1);
\draw(p1)--(p1rr);
\draw(p1)--(p1ll);
\draw(p1r)--(p1ll);
\draw(p1r)--(p1l);
\draw(p1rr)--(p1l);

\draw(p2)--(p2r);
\draw(p2r)--(p2rr);
\draw(p2rr)--(p2ll);
\draw(p2ll)--(p2l);
\draw(p2l)--(p2);
\draw(p2)--(p2rr);
\draw(p2)--(p2ll);
\draw(p2r)--(p2ll);
\draw(p2r)--(p2l);
\draw(p2rr)--(p2l);

\draw(p3)--(p3r);
\draw(p3r)--(p3rr);
\draw(p3rr)--(p3ll);
\draw(p3ll)--(p3l);
\draw(p3l)--(p3);
\draw(p3)--(p3rr);
\draw(p3)--(p3ll);
\draw(p3r)--(p3ll);
\draw(p3r)--(p3l);
\draw(p3rr)--(p3l);

\draw(p0l)--(p1r);
\draw(p1l)--(p2r);
\draw(p2l)--(p3r);
\draw(p3l)--(p0r);

\path (p0rr) edge [->,bend right=42] (p0ll);
\path (p1rr) edge [->,bend right=42] (p1ll);
\path (p2rr) edge [->,bend right=42] (p2ll);
\path (p3rr) edge [->,bend right=42] (p3ll);

\node [draw=white, fill=white] (a) at (0.2,0.4)  {$s$};
\node [draw=white, fill=white] (a) at (-0.2,0.4)  {$s$};
\draw[<-](p0)--(p2);
\draw[<-](p1)--(p3);

\node [draw=white, fill=white] (a) at (45:2.2cm)  {$s$};
\node [draw=white, fill=white] (a) at (135:2.2cm)  {$s$};
\node [draw=white, fill=white] (a) at (225:2.2cm)  {$s$};
\node [draw=white, fill=white] (a) at (315:2.2cm)  {$s$};

\node [draw=white, fill=white] (a) at (0,-2.8)  {(a)};

\end{tikzpicture}
\hspace{2.5cm}
\begin{tikzpicture}[very thick,scale=0.9]
\tikzstyle{every node}=[circle, draw=black, fill=white, inner sep=0pt, minimum width=5pt];

\node (p0) at (0,0) {};
\node (p1) at (90:1.3cm) {};
\node (p2) at (150:1.3cm) {};
\node (p3) at (210:1.3cm) {};
\node (p4) at (270:1.3cm) {};
\node (p5) at (330:1.3cm) {};
\node (p6) at (30:1.3cm) {};

\draw(p0)--(p1);
\draw(p0)--(p2);
\draw(p0)--(p3);
\draw(p0)--(p4);
\draw(p0)--(p5);
\draw(p0)--(p6);

\node [draw=white, fill=white] (a) at (90:2.15cm)  {$C_6$};
\node [draw=white, fill=white] (a) at (70:1.8cm)  {$C_3$};
\node [draw=white, fill=white] (a) at (110:1.8cm)  {$C_2$};

\node [draw=white, fill=white] (a) at (150:2.15cm)  {$C_6$};
\node [draw=white, fill=white] (a) at (130:1.8cm)  {$C_3$};
\node [draw=white, fill=white] (a) at (170:1.8cm)  {$C_2$};

\node [draw=white, fill=white] (a) at (210:2.15cm)  {$C_6$};
\node [draw=white, fill=white] (a) at (190:1.8cm)  {$C_3$};
\node [draw=white, fill=white] (a) at (230:1.8cm)  {$C_2$};

\node [draw=white, fill=white] (a) at (270:2.15cm)  {$C_6$};
\node [draw=white, fill=white] (a) at (250:1.8cm)  {$C_3$};
\node [draw=white, fill=white] (a) at (290:1.8cm)  {$C_2$};

\node [draw=white, fill=white] (a) at (30:2.15cm)  {$C_6$};
\node [draw=white, fill=white] (a) at (10:1.8cm)  {$C_3$};
\node [draw=white, fill=white] (a) at (50:1.8cm)  {$C_2$};

\node [draw=white, fill=white] (a) at (330:2.15cm)  {$C_6$};
\node [draw=white, fill=white] (a) at (310:1.8cm)  {$C_3$};
\node [draw=white, fill=white] (a) at (350:1.8cm)  {$C_2$};

\draw[->] (p1) to [out=330,in=30,looseness=15] (p1);
\draw[->] (p1) to [out=60,in=120,looseness=15] (p1);
\draw[->] (p1) to [out=150,in=210,looseness=15] (p1);

\draw[->] (p2) to [out=30,in=90,looseness=15] (p2);
\draw[->] (p2) to [out=120,in=180,looseness=15] (p2);
\draw[->] (p2) to [out=210,in=270,looseness=15] (p2);

\draw[->] (p3) to [out=90,in=150,looseness=15] (p3);
\draw[->] (p3) to [out=180,in=240,looseness=15] (p3);
\draw[->] (p3) to [out=270,in=330,looseness=15] (p3);

\draw[->] (p4) to [out=150,in=210,looseness=15] (p4);
\draw[->] (p4) to [out=240,in=300,looseness=15] (p4);
\draw[->] (p4) to [out=330,in=30,looseness=15] (p4);

\draw[->] (p5) to [out=210,in=270,looseness=15] (p5);
\draw[->] (p5) to [out=300,in=0,looseness=15] (p5);
\draw[->] (p5) to [out=30,in=90,looseness=15] (p5);

\draw[->] (p6) to [out=270,in=330,looseness=15] (p6);
\draw[->] (p6) to [out=0,in=60,looseness=15] (p6);
\draw[->] (p6) to [out=90,in=150,looseness=15] (p6);
\node [draw=white, fill=white] (a) at (0,-2.8)  {(b)};
\end{tikzpicture}
\end{center}
\vspace{-0.4cm}
\caption{(a) An example of a $\mathbb{Z}_2$-gain graph (with $\mathbb{Z}_2=\langle s \rangle$) whose covering graph is $5$-mixed connected, but not $\mathbb{Z}_2$-symmetric infinitesimally rigid. (b) An example of a connected $\mathbb{Z}_6$-gain graph (with $\mathbb{Z}_6=\langle C_6\rangle $) whose covering graph is $6$-mixed connected, but not $\mathbb{Z}_6$-symmetric infinitesimally rigid. In both (a) and (b), the orientation and edge label is omitted for every edge with gain $\mathrm{id}$.
}
\label{fig:2conn5conn}\end{figure}

The covering graph $\tilde G$ of the graph $(G,\psi)$ in Figure~\ref{fig:2conn5conn}(a) is clearly 5-mixed-connected. Moreover, to see that $\tilde{G}$ is not $\mathbb{Z}_2$-symmetric infinitesimally rigid, consider the partition of the edge set of $(G,\psi)$ consisting of the four \emph{balanced} edge sets $E_1,\ldots, E_4$  of the four $K_5$ subgraphs and the ten remaining singletons $E_5,\ldots , E_{14}$. Since $\rho(E_i)=7$ for each $i=1,\ldots, 4$, we have $\sum_{i=1}^{14}\rho(E_i)= 38< 39=2|V|-1$. Thus, by Theorems~\ref{thm:symlaman} and \ref{thm:symrank},  $\tilde{G}$ is not $\mathbb{Z}_2$-symmetric infinitesimally rigid.

To see that the  covering graph $\tilde G$ of the graph $(G,\psi)$ in Figure~\ref{fig:2conn5conn}(b) is 6-mixed-connected, observe that $\tilde G$ may be obtained from the complete bipartite graph $K_{6,6}$ with partite sets $X$ and $Y$, where $X=\{x_1,\ldots , x_6\}$ is the orbit of the `central vertex' $x$ in $(G,\psi)$ under the $6$-fold rotation, and $Y=\{y_1,\ldots , y_6\}$ is the set of the six vertices in $(G,\psi)-x$, by replacing each vertex $y_j$ in $Y$ by a complete graph $K_6$ on the vertices $y_j^{(1)},\ldots , y_j^{(6)}$, and replacing each edge $x_iy_j$ by the edges $x_iy_j^{(1)}, \ldots, x_iy_j^{(6)}$. To see that $\tilde{G}$ is not $\mathbb{Z}_6$-symmetric infinitesimally rigid, consider the partition of the edge set of $(G,\psi)$ consisting of the six edge sets $E_1,\ldots, E_6$, where $E_i$ consists of the three loops induced by $y_i$, and the remaining six singletons $E_7,\ldots, E_{12}$. Since $\rho(E_i)=1$ for all $i=1,\ldots, 12$, we have $\sum_{i=1}^{12}\rho(E_i)= 12< 13=2|V|-1$. Thus, by Theorems~\ref{thm:symlaman} and \ref{thm:symrank},  $\tilde{G}$ is not $\mathbb{Z}_2$-symmetric infinitesimally rigid.

These examples may easily be adapted to obtain analogous examples for the other groups mentioned in Theorem~\ref{thm:forcedmain}. 
An example for the dihedral group $\mathcal{C}_{3v}$ is shown in Figure~\ref{fig:dihedral}.

\begin{figure}[htp]
\begin{center}
\begin{tikzpicture}[very thick,scale=0.9]
\tikzstyle{every node}=[circle, draw=black, fill=white, inner sep=0pt, minimum width=5pt];

\node (p0) at (45:0.9cm) {};
\node (p1) at (135:0.9cm) {};
\node (p2) at (225:0.9cm) {};
\node (p3) at (315:0.9cm) {};

\node (p0r) at (25:1.3cm) {};
\node (p1r) at (115:1.3cm) {};
\node (p2r) at (205:1.3cm) {};
\node (p3r) at (295:1.3cm) {};

\node (p0l) at (65:1.3cm) {};
\node (p1l) at (155:1.3cm) {};
\node (p2l) at (245:1.3cm) {};
\node (p3l) at (335:1.3cm) {};

\node (p0rr) at (35:1.8cm) {};
\node (p1rr) at (125:1.8cm) {};
\node (p2rr) at (215:1.8cm) {};
\node (p3rr) at (305:1.8cm) {};

\node (p0ll) at (55:1.8cm) {};
\node (p1ll) at (145:1.8cm) {};
\node (p2ll) at (235:1.8cm) {};
\node (p3ll) at (325:1.8cm) {};

\node [draw=white, fill=white] (a) at (45:2.2cm)  {$s$};
\node [draw=white, fill=white] (a) at (135:2.2cm)  {$s$};
\node [draw=white, fill=white] (a) at (225:2.4cm)  {$sC_3$};
\node [draw=white, fill=white] (a) at (315:2.4cm)  {$sC_3$};

\draw(p0)--(p0r);
\draw(p0r)--(p0rr);
\draw(p0rr)--(p0ll);
\draw(p0ll)--(p0l);
\draw(p0l)--(p0);
\draw(p0)--(p0rr);
\draw(p0)--(p0ll);
\draw(p0r)--(p0ll);
\draw(p0r)--(p0l);
\draw(p0rr)--(p0l);

\draw(p1)--(p1r);
\draw(p1r)--(p1rr);
\draw(p1rr)--(p1ll);
\draw(p1ll)--(p1l);
\draw(p1l)--(p1);
\draw(p1)--(p1rr);
\draw(p1)--(p1ll);
\draw(p1r)--(p1ll);
\draw(p1r)--(p1l);
\draw(p1rr)--(p1l);

\draw(p2)--(p2r);
\draw(p2r)--(p2rr);
\draw(p2rr)--(p2ll);
\draw(p2ll)--(p2l);
\draw(p2l)--(p2);
\draw(p2)--(p2rr);
\draw(p2)--(p2ll);
\draw(p2r)--(p2ll);
\draw(p2r)--(p2l);
\draw(p2rr)--(p2l);

\draw(p3)--(p3r);
\draw(p3r)--(p3rr);
\draw(p3rr)--(p3ll);
\draw(p3ll)--(p3l);
\draw(p3l)--(p3);
\draw(p3)--(p3rr);
\draw(p3)--(p3ll);
\draw(p3r)--(p3ll);
\draw(p3r)--(p3l);
\draw(p3rr)--(p3l);

\draw(p0l)--(p1r);
\draw(p1l)--(p2r);
\draw(p2l)--(p3r);
\draw(p3l)--(p0r);

\path (p0rr) edge [->,bend right=42] (p0ll);
\path (p1rr) edge [->,bend right=42] (p1ll);
\path (p2rr) edge [->,bend right=42] (p2ll);
\path (p3rr) edge [->,bend right=42] (p3ll);

\node [draw=white, fill=white] (a) at (0.5,0.05)  {$C_3$};
\node [draw=white, fill=white] (a) at (-0.5,0.05)  {$C_3$};
\draw[<-](p0)--(p2);
\draw[<-](p1)--(p3);

\end{tikzpicture}
\end{center}
\vspace{-0.4cm}
\caption{(a) An example of a $\Gamma$-gain graph (with $\Gamma=\langle C_3, s \rangle$) whose covering graph is $5$-mixed connected, but not $\Gamma$-symmetric infinitesimally rigid. The orientation and edge label is omitted for every edge with gain $\mathrm{id}$.
}
\label{fig:dihedral}\end{figure}
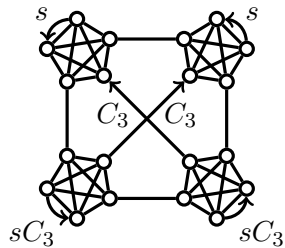

\subsection{Infinitesimal rigidity: reflection and two-fold rotational symmetry}

For every \(n\in \mathbb{N}\), it is easy to construct $\Gamma$-generic frameworks with reflection symmetry $\tau(\Gamma)=\mathcal{C}_s$ or half-turn symmetry $\tau(\Gamma)=\mathcal{C}_2$ whose underlying graphs are \(n\)-connected but that are not $\iota_1$-symmetric infinitesimally rigid. Take, for example, a realisation of the complete graph $K_n$ and its symmetric copy, and a matching between them, with all matching edges `fixed' by the non-trivial element $\gamma\in\Gamma$. (An edge $e=\{i,j\}$ is called \emph{fixed} by $\gamma$ if $\gamma(i)=j$ and $\gamma(j)=i$.) Such a framework is not $\iota_1$-symmetric infinitesimally rigid because a fixed edge in the covering graph $\tilde{G}$ corresponds to a loop in the quotient $\Gamma$-gain graph $(G,\psi)$, and such a loop is dependent by Theorem~\ref{thm:antisymlaman}.

In the following, we therefore only consider the edges of $\tilde G $ that are not fixed.
Let \(\tilde E_{\ell}\) denote the set of fixed edges in \(\tilde G\) and let \(\tilde G_{\ell}\) be \(\tilde G-\tilde E_{\ell}\).

By Theorem~\ref{thm:charcsc2}, we may combine Theorem~\ref{thm:forcedmain}, Lemma~\ref{lem:gainconn} and Theorem~\ref{thm:antimain}  to obtain the following.

\begin{thm}\label{thm:csc2maincov}
Let $(\tilde G,p)$ be a $\Gamma$-generic framework (with respect to $\theta$ and $\tau$) such that $\tau(\Gamma)$ is  $\mathcal{C}_s$ or $\mathcal{C}_{2}$. If 
\(\tilde G_{\ell}\) is 7-mixed-connected, then \((\tilde G,p)\) is  infinitesimally rigid.
\end{thm}

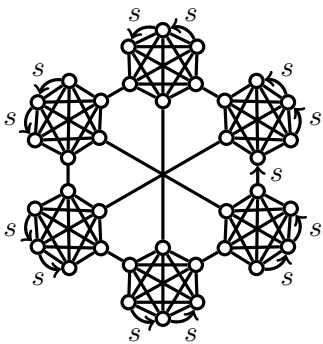
\begin{figure}[htp]
\begin{center}
\begin{tikzpicture}[very thick,scale=0.8]
\tikzstyle{every node}=[circle, draw=black, fill=white, inner sep=0pt, minimum width=5pt];

\node (p1) at (90:1.22cm) {};
\node (p2) at (70:1.6cm) {};
\node (p3) at (75:2.2cm) {};
\node (p4) at (90:2.4cm) {};
\node (p5) at (105:2.2cm) {};
\node (p6) at (110:1.6cm) {};

\node (p11) at (150:1.22cm) {};
\node (p22) at (130:1.6cm) {};
\node (p33) at (135:2.2cm) {};
\node (p44) at (150:2.4cm) {};
\node (p55) at (165:2.2cm) {};
\node (p66) at (170:1.6cm) {};

\node (p111) at (210:1.22cm) {};
\node (p222) at (190:1.6cm) {};
\node (p333) at (195:2.2cm) {};
\node (p444) at (210:2.4cm) {};
\node (p555) at (225:2.2cm) {};
\node (p666) at (230:1.6cm) {};

\node (p1111) at (270:1.22cm) {};
\node (p2222) at (250:1.6cm) {};
\node (p3333) at (255:2.2cm) {};
\node (p4444) at (270:2.4cm) {};
\node (p5555) at (285:2.2cm) {};
\node (p6666) at (290:1.6cm) {};

\node (p11111) at (330:1.22cm) {};
\node (p22222) at (310:1.6cm) {};
\node (p33333) at (315:2.2cm) {};
\node (p44444) at (330:2.4cm) {};
\node (p55555) at (345:2.2cm) {};
\node (p66666) at (350:1.6cm) {};

\node (p111111) at (30:1.22cm) {};
\node (p222222) at (10:1.6cm) {};
\node (p333333) at (15:2.2cm) {};
\node (p444444) at (30:2.4cm) {};
\node (p555555) at (45:2.2cm) {};
\node (p666666) at (50:1.6cm) {};

\draw(p1)--(p2);
\draw(p2)--(p3);
\draw(p3)--(p4);
\draw(p4)--(p5);
\draw(p5)--(p6);
\draw(p1)--(p6);
\draw(p1)--(p3);
\draw(p1)--(p4);
\draw(p1)--(p5);
\draw(p2)--(p4);
\draw(p2)--(p5);
\draw(p2)--(p6);
\draw(p3)--(p5);
\draw(p3)--(p6);
\draw(p4)--(p6);

\draw(p11)--(p22);
\draw(p22)--(p33);
\draw(p33)--(p44);
\draw(p44)--(p55);
\draw(p55)--(p66);
\draw(p11)--(p66);
\draw(p11)--(p33);
\draw(p11)--(p44);
\draw(p11)--(p55);
\draw(p22)--(p44);
\draw(p22)--(p55);
\draw(p22)--(p66);
\draw(p33)--(p55);
\draw(p33)--(p66);
\draw(p44)--(p66);

\draw(p111)--(p222);
\draw(p222)--(p333);
\draw(p333)--(p444);
\draw(p444)--(p555);
\draw(p555)--(p666);
\draw(p111)--(p666);
\draw(p111)--(p333);
\draw(p111)--(p444);
\draw(p111)--(p555);
\draw(p222)--(p444);
\draw(p222)--(p555);
\draw(p222)--(p666);
\draw(p333)--(p555);
\draw(p333)--(p666);
\draw(p444)--(p666);

\draw(p1111)--(p2222);
\draw(p2222)--(p3333);
\draw(p3333)--(p4444);
\draw(p4444)--(p5555);
\draw(p5555)--(p6666);
\draw(p1111)--(p6666);
\draw(p1111)--(p3333);
\draw(p1111)--(p4444);
\draw(p1111)--(p5555);
\draw(p2222)--(p4444);
\draw(p2222)--(p5555);
\draw(p2222)--(p6666);
\draw(p3333)--(p5555);
\draw(p3333)--(p6666);
\draw(p4444)--(p6666);

\draw(p11111)--(p22222);
\draw(p22222)--(p33333);
\draw(p33333)--(p44444);
\draw(p44444)--(p55555);
\draw(p55555)--(p66666);
\draw(p11111)--(p66666);
\draw(p11111)--(p33333);
\draw(p11111)--(p44444);
\draw(p11111)--(p55555);
\draw(p22222)--(p44444);
\draw(p22222)--(p55555);
\draw(p22222)--(p66666);
\draw(p33333)--(p55555);
\draw(p33333)--(p66666);
\draw(p44444)--(p66666);

\draw(p111111)--(p222222);
\draw(p222222)--(p333333);
\draw(p333333)--(p444444);
\draw(p444444)--(p555555);
\draw(p555555)--(p666666);
\draw(p111111)--(p666666);
\draw(p111111)--(p333333);
\draw(p111111)--(p444444);
\draw(p111111)--(p555555);
\draw(p222222)--(p444444);
\draw(p222222)--(p555555);
\draw(p222222)--(p666666);
\draw(p333333)--(p555555);
\draw(p333333)--(p666666);
\draw(p444444)--(p666666);

\path (p3) edge [->,bend right=42] (p4);
\path (p4) edge [->,bend right=42] (p5);

\path (p33) edge [->,bend right=42] (p44);
\path (p44) edge [->,bend right=42] (p55);

\path (p333) edge [->,bend right=42] (p444);
\path (p444) edge [->,bend right=42] (p555);

\path (p3333) edge [->,bend right=42] (p4444);
\path (p4444) edge [->,bend right=42] (p5555);

\path (p33333) edge [->,bend right=42] (p44444);
\path (p44444) edge [->,bend right=42] (p55555);

\path (p333333) edge [->,bend right=42] (p444444);
\path (p444444) edge [->,bend right=42] (p555555);

\draw(p6)--(p22);
\draw(p66)--(p222);
\draw(p666)--(p2222);
\draw(p6666)--(p22222);
\draw[->](p66666)--(p222222);
\draw(p666666)--(p2);

\draw(p1)--(p1111);
\draw(p11)--(p11111);
\draw(p111)--(p111111);

\node [draw=white, fill=white] (a) at (80:2.7cm)  {$s$};
\node [draw=white, fill=white] (a) at (100:2.7cm)  {$s$};

\node [draw=white, fill=white] (a) at (140:2.7cm)  {$s$};
\node [draw=white, fill=white] (a) at (160:2.7cm)  {$s$};

\node [draw=white, fill=white] (a) at (200:2.7cm)  {$s$};
\node [draw=white, fill=white] (a) at (220:2.7cm)  {$s$};

\node [draw=white, fill=white] (a) at (260:2.7cm)  {$s$};
\node [draw=white, fill=white] (a) at (280:2.7cm)  {$s$};

\node [draw=white, fill=white] (a) at (320:2.7cm)  {$s$};
\node [draw=white, fill=white] (a) at (340:2.7cm)  {$s$};

\node [draw=white, fill=white] (a) at (20:2.7cm)  {$s$};
\node [draw=white, fill=white] (a) at (40:2.7cm)  {$s$};

\node [draw=white, fill=white] (a) at (0:1.89cm)  {$s$};

\end{tikzpicture}
\end{center}
\vspace{-0.4cm}
\caption{A $\mathbb{Z}_2$-gain graph (with $\mathbb{Z}_2=\langle s \rangle$) whose covering graph is $6$-mixed connected, but not $\iota_1$-symmetric infinitesimally rigid. The orientation and edge label is omitted for every edge with gain $\mathrm{id}$.
}
\label{fig:6conn}\end{figure}
The example in Figure~\ref{fig:6conn} shows that this result is best possible. It is straightforward to check that the covering graph $\tilde G$ of the graph $(G,\psi)$ in Figure~\ref{fig:6conn} is 6-mixed-connected. To see that $\tilde{G}$ is not $\iota_1$-symmetric infinitesimally rigid (and hence not infinitesimally rigid), consider the partition of the edge set of $(G,\psi)$ consisting of the six \emph{unbalanced} edge sets $E_1,\ldots, E_6$, each of which comprises the edges of a balanced $K_6$ subgraph plus the two induced edges with gain $s$, and the 9 remaining singletons $E_7,\ldots , E_{15}$. Since $\mu(E_i)=10$ for each $i=1,\ldots, 6$, we have $\sum_{i=1}^{15}\mu(E_i)= 69< 70=2|V|-2$. Thus, by Theorems~\ref{thm:antisymlaman} and \ref{lem:2,3,2},  $\tilde{G}$ is not $\iota_1$-symmetric infinitesimally rigid.

\subsection{Infinitesimal rigidity: rotational symmetry of order $k\geq 3$ }
It was shown in \cite{ST} that a generic \(\C_3\)-symmetric framework is \(\C_3\)-symmetric infinitesimally rigid if and only if it is infinitesimally rigid. If we combine this with Theorem \ref{thm:forcedmain} we get the following sufficient condition:

\begin{cor}\label{thm:incidentalc3}
Let $(\tilde G,p)$ be a $\C_3$-generic framework (with respect to $\theta$ and $\tau$).  If \(\tilde G\) is 6-mixed-connected, then \((\tilde G,p)\) is infinitesimally rigid.
\end{cor}

For rotational groups $\C_k$ for odd \(k\) and \(5\leq k <1000\) we can prove the following similar result which follows from Lemma~\ref{lem:gainconn} and Theorem~\ref{thm:incidentalc3g}:

\begin{thm}\label{thm:incidentalc3}
Let $(\tilde G,p)$ be a $\C_k$-generic framework (with respect to $\theta$ and $\tau$) where  \(5\leq k <1000\)  is odd, and let $(G,\psi)$ be the quotient
$\Gamma$-gain graph of $\tilde G$. Suppose \(\tilde G\) is 6-mixed-connected. If $k\geq 7$ suppose further that \((G,\psi)\) is 2-edge-connected. Then \((\tilde G,p)\) is  infinitesimally rigid.
\end{thm}

\section{Mixed-connectivity versus gain-mixed-connectivity}\label{sec:gainmixed}

We now define a notion of mixed-connectivity for quotient $\Gamma$-gain graphs $(G,\psi)$.
The main result of this section, Lemma \ref{lem:gainconn}, relates this notion, called $n$-gain-mixed-connectivity of $(G,\psi)$, with $n$-mixed-connectivity of the corresponding covering graph $\tilde G$.
The notion of $n$-gain-mixed-connectivity is introduced using definitions similar to the ones that can be found in \cite{kst}.

 Let \((G,\psi)\) be a \(\Gamma\)-gain graph where \(G=(V,E)\) and let \(H=(V',E')\) be a subgraph of \(G\). Suppose that \(H\) is a connected component of \(G-U-D\) for some \(U\subseteq V\), \(D\subseteq E\). Suppose further that if \(G-U-D=H\) then \(\langle E'\rangle_{\psi}\) is a proper subgroup of \(\Gamma\).

Consider the subgraph \(\langle E'\rangle_{\psi}H\) of \(\tilde G\) with vertex set \(\langle E'\rangle_{\psi} V'=\{ \gamma i\, | \, \gamma \in \langle E'\rangle_{\psi}, i\in V'  \}\) and edge set \(\langle E'\rangle_{\psi} E'=\{ \{\gamma i, \gamma \psi(e) j\}\, | \, \gamma \in \langle E'\rangle_{\psi}, e=(i,j)\in E'  \}\).
We will first construct two sets \((W,F)\), \(W\subseteq \tilde V\) and \(F\subseteq \tilde E\), such that one connected component in \(\tilde G-W-F\) is \(\langle E'\rangle_{\psi}H\). Clearly, \(\langle E'\rangle_{\psi}H\) is connected by the connectivity of \(H\). For every edge in \(D\) and every vertex in \(U\) we will take an appropriate corresponding subset of edges or vertices in the covering graph, respectively.

As we are interested in `separating sets' that are minimal we can assume that at least one endpoint of every edge in \(D\) is incident with a vertex in \(V'\) and there is at least one edge in every edge orbit that has an endpoint in \(V(\langle E'\rangle_{\psi}H)\) and one outside of it. Let \(v\in V'\) be a vertex incident with \(e\) in \(G\), and let $o(\psi(e))$ denote the order of the group element $\psi(e)$. Further, define \(e_H\) to be the edges from the orbit of \(e\) that are incident with a vertex from \(V(\langle E'\rangle_{\psi}H)\). By the observation above, we have the following cases:

Case 1: \(e\) is a loop and \(\langle E'\rangle_{\psi}\neq\langle E'+e\rangle_{\psi}\): If \(o(\psi(e))=2\) then
\(|e_H|=|\langle E'\rangle_{\psi}|\) as every vertex from the orbit of \(v\) in \(V(\langle E'\rangle_{\psi}H)\) is incident to exactly one edge in \(e_H\). If \(o(\psi(e))\geq3\) then \(|e_H|=2|\langle E'\rangle_{\psi}|\) as now instead of one there are two edges incident with each vertex from the orbit of \(v\), one that connects \(\gamma v\) with \(\gamma\psi(e) v\) and one that connects \(\gamma v\) with \(\gamma\psi(e)^{-1} v\).

Case 2: \(e\) is a non-loop edge with exactly one endpoint in \(V'\): Clearly \(|e_H|=|\langle E'\rangle_{\psi}|\) holds, as for each copy of \(H\) there is one edge in the orbit of \(e\) that has an endpoint in it.

Case 3: \(e\) is a non-loop edge with both endpoints in \(V'\) and \(\langle E'\rangle_{\psi}\neq\langle E'+e\rangle_{\psi}\): In this case, we clearly have \(|e_H|=2|\langle E'\rangle_{\psi}|\).

For a vertex \(v\notin V'\), let \(H_v=(V'+v,E_v')\) denote the subgraph of \(G\) that contains the edges in \(E'\) and every edge between \(v\) and \(V'\) (but not the loops incident with \(v\) if any exist). Let \(v_H=\langle E_v'\rangle_{\psi}v\). These are the vertices in the orbit of \(v\) that are incident to a vertex in \(V(\langle E'\rangle_{\psi}H)\).

We say that a subgraph \(H=(V',E')\) of \(G\) (with \(\psi\) restricted to \(E'\)) is a \emph{\(k\)-block} if
\begin{itemize}
  \item there is a pair \((U,D)\), \(U\subseteq V\), \(D\subseteq E\), so that \(H\) is a connected component of \(G-U-D\)
  such that \[2\sum_{v\in U}|v_H|+\sum_{e\in D}|e_H|=k\] holds;
  \item if \(G-U-D=H\), then \(\langle E'\rangle_{\psi}\) is a proper subgroup of \(\Gamma\).
\end{itemize}
\((G,\psi)\) is said to be \emph{$n$-gain-mixed-connected} if it has no \(k\)-block with \(k\leq n-1\).

For a \(k\)-block \((U,D)\) of \((G,\psi)\) we may define the corresponding \emph{symmetric separation} of \(\tilde{G}\). Let \(H=(V',E')\) be a connected component of \(G-U-D\). Then
the corresponding symmetric separation is the pair \((U_H,D_H)\) where
\[U_H=\bigcup_{v\in U}v_H\hbox{, }D_H=\bigcup_{e\in D}e_H.\]
We will show in the proof of Lemma \ref{lem:gainconn} that this set indeed disconnects \(\tilde{G}\).

\begin{lem}\label{lem:gainconn}
\begin{enumerate}
  \item[(a)] Suppose that \(\tilde{G}\) is $n$-mixed-connected. Then \((G,\psi)\) is $n$-gain-mixed-connected.
  \item[(b)] Suppose that \((G,\psi)\) is $n$-gain-mixed-connected. Then for every symmetric separation \((W,F)\) of \(\tilde{G}\), \(2|W|+|F|\geq n\) holds.
\end{enumerate}
\end{lem}

\begin{proof}
Take a \(k\)-block \(H=(V',E')\) in \(G\) with
corresponding sets \(U,\ D\). The graph
\(\tilde{G}-U_H-D_H\) is not connected, as by the definitions of \(v_H\) and \(e_H\) there is no edge leaving \(\langle E'\rangle_{\psi}V'\), and this set is a proper subset  of \(\tilde V-U_H\).

We also have \(2|U_H|+|D_H|=2\cup_{v\in U}|v_H|+\cup_{e\in D}|e_H|=k\), as \(H=(V',E')\) is a \(k\)-block.

(a) Suppose for a contradiction that \((G,\psi)\) is not $n$-gain-mixed-connected, that is, it has a \(k\)-block \(H\) for some \(k\leq n-1\). Then, by the calculation above, we can deduce that \(\tilde{G}\) is not $n$-mixed-connected, which is a contradiction.

(b) Take a symmetric separation \((W,F)\) of \(\tilde{G}\). There exists a \(k\)-block \(H\) and sets \(U\subseteq V\), \(D\subseteq E\) for which \(W=U_H\) and \(F=D_H\). Using the same calculation again we get that \(2|W|+|F|\geq n\) must hold by the $n$-gain-mixed-connectivity of \((G,\psi)\). This completes the proof.
\end{proof}

\section{Sufficient connectivity conditions for rigidity in term of gain graphs}\label{sec:suffforced}

In this chapter, we establish sufficient criteria for the rigidity of forced symmetric and incidentally smmetric frameworks in the plane. These are given purely in terms of connectivity conditions for the corresponding quotient $\Gamma$-gain graphs. Together with Lemma~\ref{lem:gainconn}, they imply the main results in Section~\ref{sec:coveringsufficient}.

\subsection{Symmetric covers}

To prove our results, we first  need the following definitions. For a graph $\tilde{G}=(\tilde V, \tilde E)$ and disjoint sets $X,Y\subseteq \tilde V$, we let $d_{\tilde G}(X,Y)$ denote the number of edges between $X$ and $Y$, and we let $d_{\tilde G}(X):=d_{\tilde G}(X,\tilde V \setminus X)$.  In particular, $d_{\tilde G}(v)=d_{\tilde G}(\{v\})$ for $v\in \tilde V$.

A set \(\X\subseteq 2^{\tilde V}\) is called a \emph{cover} of \(\tilde G\) if \(\tilde E=\cup_{X\in\X}\tilde E(X)\).
For a partition \(\mathcal{P}=\{E_1,\dots,E_s\}\) of the edge set \(E\) of the quotient $\Gamma$-gain graph $(G,\psi)$ of $\tilde G$, we define a cover \(\X\) of \(\tilde G\) as follows.

Consider \(E_i\in\mathcal{P}\). To simplify notation we will denote \(\langle E_i\rangle_{\psi}\) by \(\Gamma_i\). There exists a labelling \(\psi_i\) equivalent to \(\psi\) such that the label of every edge
in \(E_i\) is an element of \(\Gamma_i\) \cite{jkt}. Choose a representing element of every vertex orbit of \(\tilde G\) such that the chosen elements define \(\psi_i\).
Let \(\tilde V_i\subseteq \tilde V\) contain those of the representing elements which correspond to the vertex orbits of \(V(E_i)\).
Then the vertex set corresponding to \(E_i\) in $\tilde G$ is \(X_i=\Gamma_i\tilde V_i\). Every vertex set \(\gamma X_i\) with \(\gamma\in\Gamma\) belongs to \(\X\).
(Note that these sets are not necessarily pairwise distinct.)
Thus every \(E_i\in\mathcal{P}\) defines \(|\Gamma|/|\Gamma_i|\) vertex sets in \(\X\) and \(\X=\{X\subseteq V:X=\gamma X_i\hbox{ for some }\gamma\in\Gamma,E_i\in\mathcal{P}\}\).
We will call a cover of \(\tilde E\) that can be obtained from a partition of \(E\) by applying the above process a \emph{symmetric cover}.

Consider the gain graph in Figure \ref{fig:gain}(b), for example. We define the partition \(\mathcal{P}=\{E_1,E_2,E_3\}\) as follows. Let \(E_1\) contain the parallel pair of edges, and let
\(E_2\), \(E_3\) be the two singleton sets containing the loop with label \(C_3\) and \(sC_3^2\), respectively. The representative vertices are \(1\) and \(2\)
and \(\tilde V_1=\{1,2\}\), \(\tilde V_2=\tilde V_3=\{1\}\). The groups of the edge sets are \(\Gamma_1=\langle s\rangle\),
\(\Gamma_2=\langle C_3\rangle\), \(\Gamma_3=\langle sC_3^2\rangle\). So \(X_1=\{1,2,s1,s2\}\),
\(X_2=\{1,C_31,C_3^21\}\), \(X_3=\{1,sC_3^21\}\). Finally \(\mathcal{X}=\left\{X_1,C_3X_1,C_3^2X_1,X_2,sX_2,X_3,C_3X_3,C_3^2X_3\right\}\).

We will use the following notation. For \(X\in\X\) let \(E_X=E_i\) for which there is a \(\gamma\in\Gamma\) with \(\gamma X_i=X\). Further, we let \(\X_u=\{X\in\X:|\Gamma_X|\geq 4\}\) and  \(\X_3=\{X\in\X:|X|\geq3\}\). Finally, we let \(\X_2=\{X\in\X : E_X\hbox{ is unbalanced}\}\).

We remark that every edge orbit of a $\Gamma$-symmetric graph contains at most $|\Gamma|$ edges. This implies that the 2-edge-connectivity of \(G\) is a corollary of the 6-mixed-connectivity of
$(\tilde G,p)$ if \(|\Gamma|<6\).

\subsection{Forced symmetric rigidity}

We show the following main theorem, which in turn implies Theorem~\ref{thm:forcedmain}.

\begin{thm}\label{thm:forcedgainmain}
Let $(\tilde G,p)$ be a $\Gamma$-generic framework (with respect to $\theta$ and $\tau$) such that $\tau(\Gamma)$ is  $\mathcal{C}_s$, $\mathcal{C}_k$ or $\mathcal{C}_{(2k+1)v}$,
and let $(G,\psi)$ be the quotient $\Gamma$-gain graph of $\tilde G$. Suppose \((G,\psi)\) is 6-gain-mixed-connected. If \(|\Gamma|\geq6\) then suppose further that \((G,\psi)\) is 2-edge-connected. Then \((\tilde G,p)\) is $\Gamma$-symmetric infinitesimally rigid.
\end{thm}

To prove this result, we need the following key lemma.

\begin{lem}\label{lem:coverlower}
Suppose that \((G,\psi)\) is a 6-gain-mixed-connected $\Gamma$-gain graph. If \(|\Gamma|\geq6\) then suppose further that \((G,\psi)\) is 2-edge-connected. Then for every symmetric cover \(\X\), we have
\[\sum_{X\in\X}(2|X|-3)\geq2|\tilde V|+\sum_{X\in\X_3\cap\X_u}(|\Gamma_X|-3).\]
\end{lem}

\begin{proof}
Let \(F=\bigcup_{X\in\X_3}\tilde E(X)\). With this notation, we have
$\sum_{X\in \X}(2|X|-3)=\sum_{X\in \X_3}(2|X|-3)+|\tilde E-F|.$

Let \(Y_X=X\bigcap\bigcup_{X'\in\X_3,X'\neq X}X'\) and  \(\X'=\{X\in\X_3:X\neq Y_X\}\). Then

\[|\tilde E-F|\geq\frac{1}{2}\left(\sum_{X\in\X'}d_{\tilde G-Y_X}(X-Y_X)+\sum_{v\in \tilde V-\tilde V(\X_3)}d_{\tilde G}(v)\right).\]

For an arbitrary \(X\in\mathcal{X}'\) the pair \(Y_X\) and the edge set \(E(X-Y_X,\tilde{V}-X)\) between \(X-Y_X\) and \(\tilde{V}-X\) is a symmetric separation of \(\tilde{G}\). Thus by Lemma \ref{lem:gainconn}(b) \(2|Y_X|+d_{\tilde G-Y_X}(X-Y_X)\geq6\). After reordering this we get \(d_{\tilde G-Y_X}(X-Y_X)\geq\max\{6-2|Y_X|,0\}\) for all \(X\in\X'\). Suppose first that for \(X\in\X\) we have \(Y_X\neq\emptyset\). Observe that if \(v\in Y_X\) for some \(v\in \tilde V\), then \(\gamma v\in Y_X\) for every \(\gamma\in\Gamma_X\). Thus \(|Y_X|\geq|\Gamma_X|\), and if \(|\Gamma_X|\geq4\), then \(d_{\tilde G-Y_X}(X-Y_X)\geq0\geq6-2|Y_X|+(2|\Gamma_X|-6)\). If \(Y_X=\emptyset\), then the same inequality holds, since $d_{\tilde G}(X)\geq 2|\Gamma_X|$ by the 2-edge-connectivity of \((G,\psi)\). Thus

\[|\tilde E-F|\geq3|\X'|-\sum_{X\in\X'}|Y_X|+\sum_{X\in\X'\cap\X_u}(|\Gamma_X|-3)+3(|\tilde V|-|\tilde V(\X_3)|).\]

Using this we have

\[\sum_{X\in \X_3}(2|X|-3)+|\tilde E-F|\geq\]
\[\geq2\sum_{X\in\X_3}|X|-3|\X_3|+3|\X'|-\sum_{X\in\X'}|Y_X|+\sum_{X\in\X'\cap\X_u}(|\Gamma_X|-3)+3(|\tilde V|-|\tilde V(\X_3)|).\]

For every \(X\in\X_3-\X'\) we have \(|Y_X|=|X|\geq\max\{3,|\Gamma_X|\}\). Thus

\[3|\X_3|-3|\X'|=3|\X_3-\X'|\leq\sum_{X\in\X_3-\X'}|Y_X|-\sum_{X\in(\X_3-\X')\cap\X_u}(|\Gamma_X|-3).\]

Using that \(|\tilde V|-|\tilde V(\X_3)|\geq0\) this implies

\[\sum_{X\in \X_3}(2|X|-3)+|\tilde E-F|-2|\tilde V|\geq\]
\[\geq2\sum_{X\in \X_3}|X|-\sum_{X\in\X_3-\X'}|Y_X|-\sum_{X\in\X'}|Y_X|+\sum_{X\in\X_3\cap\X_u}(|\Gamma_X|-3)-2|\tilde V(\X_3)|+(|\tilde V|-|\tilde V(\X_3)|)\]
\[\geq2\sum_{X\in \X_3}(|X|-|Y_X|)+\sum_{X\in\X_3}|Y_X|+\sum_{X\in\X_3\cap\X_u}(|\Gamma_X|-3)-2|\tilde V(\X_3)|.\]

\(2\sum_{X\in \X_3}(|X|-|Y_X|)\) is twice the number of vertices in \(\tilde V(\X_3)\) contained by exactly one \(X\). In \(\sum_{X\in\X_3}|Y_X|\) every vertex contained in some \(Y_X\) with \(X\in\X_3\) is counted at least twice. Thus \(2\sum_{X\in \X_3}(|X|-|Y_X|)+\sum_{X\in\X_3}|Y_X|\geq2|\tilde V(\X_3)|\). Hence

\[\sum_{X\in\X}(2|X|-3)=\sum_{X\in\X_3}(2|X|-3)+|\tilde E-F|\geq2|\tilde V|+\sum_{X\in\X_3\cap\X_u}(|\Gamma_X|-3)\]

as we claimed.
\end{proof}

We are now ready to prove Theorem~\ref{thm:forcedgainmain}.

\begin{proof}
Suppose for a contradiction that \((G,\psi)\) is 6-gain-mixed-connected and 2-edge-connected, but \((\tilde G,p)\) is not $\Gamma$-symmetric infinitesimally rigid. Equivalently, the edge set \(E\) of  \((G,\psi)\) has a partition \(\mathcal{P}=\{E_1,\dots,E_s\}\) with \(\sum_{i=1}^{s}\rho(E_X)\leq 2|V|-2\) if \(\tilde G\) has rotational or reflectional symmetry or \(\sum_{i=1}^{s}\rho(E_X)\leq 2|V|-1\) if \(\tilde G\) has dihedral symmetry.

Construct the symmetric cover \(\X\) of \(\tilde G\) from \(\mathcal{P}\). By the construction of \(\X\)
\begin{equation}\label{eq:numvertices}
|X|=|\Gamma_X||V(E_X)|\hbox{ for every }X\in\X,
\end{equation}

from which we obtain

\begin{equation}\label{eq:2pi3}
2|V(E_X)|-1\geq\frac{2|X|-3}{|\Gamma_X|}\hbox{ if }2\leq|\Gamma_X|\leq3,
\end{equation}

\begin{equation}
2|V(E_X)|-1=\frac{(2|X|-3)-(|\Gamma_X|-3)}{|\Gamma_X|}\hbox{ if }|\Gamma_X|\geq4\hbox{ and \(\Gamma_X\) is cyclic,}
\end{equation}

\begin{equation}
2|V(E_X)|\geq\frac{2|X|-3}{|\Gamma_X|}\hbox{ if \(\Gamma_X\) is dihedral.}
\end{equation}

Let \(\mathcal{P}_b=\{E_i:\Gamma_i\hbox{ is balanced}\}\), \(\mathcal{P}_c=\{E_i:\Gamma_i\hbox{ is cyclic but not balanced}\}\), and \(\mathcal{P}_d=\{E_i:\Gamma_i\hbox{ is dihedral}\}\).

Using the observations above we obtain
\[|\Gamma|\sum_{i=1}^{t}\rho(E_X)
=|\Gamma|\left(\sum_{E_X\in\mathcal{P}_b}(2|V(E_X)|-3)+\sum_{E_X\in\mathcal{P}_c}(2|V(E_X)|-1)+\sum_{E_X\in\mathcal{P}_d}2|V(E_X)|\right)\]
\[\geq|\Gamma|\sum_{E_X\in\mathcal{P}_b}(2|X|-3)+\sum_{E_X\in\mathcal{P}_c}\frac{|\Gamma|}{|\Gamma_X|}(2|X|-3)-|\Gamma|\sum_{E_X\in\mathcal{P}_c,|\Gamma_X|\geq4}\frac{|\Gamma_X|-3}{|\Gamma_X|}+
\sum_{E_X\in\mathcal{P}_d}\frac{|\Gamma|}{|\Gamma_X|}(2|X|-3)\]
\[=\sum_{X\in\X}(2|X|-3)-|\Gamma|\sum_{E_X\in\mathcal{P}_c,|\Gamma_X|\geq4}\frac{|\Gamma_X|-3}{|\Gamma_X|}\]
\[\geq2|\tilde V|+\sum_{X\in\X_3\cap\X_u}(|\Gamma_X|-3)-|\Gamma|\sum_{E_X\in\mathcal{P}_c,|\Gamma_X|\geq4}\frac{|\Gamma_X|-3}{|\Gamma_X|}
\geq 2|\tilde V|=2|\Gamma||V|,\]
where the penultimate inequality follows from Lemma \ref{lem:coverlower}. This is a contradiction which completes the proof.
\end{proof}

\subsection{Infinitesimal rigidity: reflection and two-fold rotational symmetry} \label{subsec:csc2}

Let \((G,\psi)\) be a gain graph. Denote the gain graph by deleting all the loops from \((G,\psi)\) by \((G_{\ell},\psi)\). We show the following main theorem, which in turn implies Theorem~\ref{thm:csc2maincov}:

\begin{thm}\label{thm:antimain}
Let $(\tilde G,p)$ be a $\Gamma$-generic framework (with respect to $\theta$ and $\tau$) such that $\tau(\Gamma)$ is  $\mathcal{C}_s$ or $\mathcal{C}_2$. If \((G_{\ell},\psi)\) is 7-gain-mixed-connected, then \((\tilde G,p)\) is $\iota_1$-symmetric infinitesimally rigid.
\end{thm}

To prove this result, we need the following key lemma.

\begin{lem}\label{lem:CsC2}
Suppose that \((G_{\ell},\psi)\) is 7-gain-mixed-connected. Then for every symmetric cover \(\X\), we have
\[\sum_{X\in\X}(2|X|-3)\geq2|\tilde V|+|\X_2|.\]
\end{lem}
\begin{proof}
Let \(F=\bigcup_{X\in\X_3}\tilde E(X)\). With this notation, we have
\[\sum_{X\in \X}(2|X|-3)=\sum_{X\in \X_3}(2|X|-3)+|\tilde E-F|.\]

Let \(Y_X=X\bigcap\bigcup_{X'\in\X_3,X'\neq X}X'\) and \(\X'=\{X\in\X_3:X\neq Y_X\}\). Then
\[|\tilde E-F|\geq\frac{1}{2}\left(\sum_{X\in\X'}d_{\tilde G_{\ell}-Y_X}(X-Y_X)+\sum_{v\in \tilde V-\tilde V(\X_3)}d_{\tilde G}(v)\right).\]

By the 7-gain-mixed-connectivity of \((G_{\ell},\psi)\) we have \(d_{\tilde G_{\ell}-Y_X}(X-Y_X)\geq\max\{7-2|Y_X|,0\}\) for all \(X\in\X'\). Observe that if \(X\in\X_2\) then \(|Y_X|\) has to be even. Suppose first that for \(X\in\X'\cap\X_2\) we have \(|Y_X|\geq4\). Then \(d_{\tilde G_{\ell}-Y_X}(X-Y_X)\geq6-2|Y_X|+2\). If \(|Y_X|=2\) then \(d_{\tilde G_{\ell}-Y_X}(X-Y_X)\geq3=7-2|Y_X|\) but as \(d_{\tilde G_{\ell}-Y_X}(X-Y_X)\) must be even, we can also deduce \(d_{\tilde G_{\ell}-Y_X}(X-Y_X)\geq6-2|Y_X|+2\). If \(Y_X=\emptyset\) then \(d_{\tilde G_{\ell}-Y_X}(X-Y_X)\geq7=7-2|Y_X|\) and again by the parity argument \(d_{\tilde G_{\ell}-Y_X}(X-Y_X)\geq6-2|Y_X|+2\). Thus
\[|\tilde E-F|\geq3|\X'|-\sum_{X\in\X'}|Y_X|+|\X'\cap\X_2|+3(|\tilde V|-|\tilde V(\X_3)|).\]

Using this we have
\[\sum_{X\in \X_3}(2|X|-3)+|\tilde E-F|\]
\[\geq2\sum_{X\in\X_3}|X|-3|\X_3|+3|\X'|-\sum_{X\in\X'}|Y_X|+|\X'\cap\X_2|+3(|\tilde V|-|\tilde V(\X_3)|).\]

For every \(X\in\X_2-\X'\) we have \(|Y_X|=|X|\geq4\) and for every \(X\in\X_3-\X'\) we have \(|Y_X|=|X|\geq3\). Thus
\[3|\X_3|-3|\X'|=3|\X_3-\X'|\leq\sum_{X\in\X_3-\X'}|Y_X|-|\X_2-\X'|.\]

Using  \(|\tilde V|-|\tilde V(\X_3)|\geq0\) this implies
\[\sum_{X\in \X_3}(2|X|-3)+|\tilde E-F|-2|\tilde V|\]
\[\geq2\sum_{X\in \X_3}|X|-\sum_{X\in\X_3-\X'}|Y_X|-\sum_{X\in\X'}|Y_X|+|\X_2|-2|\tilde V(\X_3)|+(|\tilde V|-|\tilde V(\X_3)|)\]
\[\geq2\sum_{X\in \X_3}(|X|-|Y_X|)+\sum_{X\in\X_3}|Y_X|+|\X_2|-2|\tilde V(\X_3)|.\]

\(2\sum_{X\in \X_3}(|X|-|Y_X|)\) is twice the number of vertices in \(\tilde V(\X_3)\) contained by exactly one \(X\). In \(\sum_{X\in\X_3}|Y_X|\) every vertex contained in some \(Y_X\) with \(X\in\X_3\) is counted at least twice. Thus \(2\sum_{X\in \X_3}(|X|-|Y_X|)+\sum_{X\in\X_3}|Y_X|\geq2|\tilde V(\X_3)|\). Hence
\[\sum_{X\in\X}(2|X|-3)=\sum_{X\in\X_3}(2|X|-3)+|\tilde E-F|\geq2|\tilde V|+|\X_2|,\]
as we claimed.
\end{proof}

We are now ready to prove Theorem~\ref{thm:antimain}.

\begin{proof}
Suppose for a contradiction that \((G_{\ell},\psi)\) is 7-gain-mixed-connected but \((\tilde G,p)\) is not $\iota_1$-symmetric infinitesimally rigid. Equivalently, the edge set \(E\) of the quotient $\Gamma$-gain graph \((G,\psi)\) has a partition \(\mathcal{P}=\{E_1,\dots,E_s\}\) with \(\sum_{i=1}^{s}\mu(E_X)\leq 2|V|-3\).
Construct the symmetric cover \(\X\) of \(\tilde G\) from \(\mathcal{P}\). Let \(\mathcal{P}_b=\{E_i:\Gamma_i\hbox{ is balanced}\}\), \(\mathcal{P}_c=\{E_i:\Gamma_i\hbox{ is cyclic but not balanced}\}\).
We have
\[2\sum_{i=1}^s\mu(E_X)=2\left(\sum_{E_X\in\mathcal{P}_b}(2|V(E_X)|-3)+\sum_{E_X\in\mathcal{P}_c}(2|V(E_X)|-2)\right)\]
\[=2\sum_{E_X\in\mathcal{P}_b}(2|X|-3)+\sum_{E_X\in\mathcal{P}_c}(2|X|-4)=\sum_{X\in\X}(2|X|-3)-|\X_2|\geq2|\tilde V|=4|V|,\]
where the last inequality follows from Lemma \ref{lem:CsC2}. This is a contradiction which completes the proof.
\end{proof}

\subsection{Infinitesimal rigidity: rotational symmetry of order $k\geq 3$}

Finally, we prove the following main result for the groups $\C_k$, $5 \leq k< 1000$ odd, which in turn implies Theorem~\ref{thm:incidentalc3}.

\begin{thm}\label{thm:incidentalc3g}
Let $(\tilde G,p)$ be a $\C_k$-generic framework (with respect to $\theta$ and $\tau$) where  \(5\leq k <1000\)  is odd, and let $(G,\psi)$ be the quotient
$\Gamma$-gain graph of $\tilde G$. Suppose \(\tilde G\) is 6-mixed-connected. If $k\geq 7$ suppose further that \((G,\psi)\) is 2-edge-connected. Then 
\((\tilde G,p)\) is  infinitesimally rigid.
\end{thm}
\begin{proof}
Take  an arbitrary partition \(E_1,\dots,E_s\) of \(E(G)\). By Theorem \ref{thm:antisymlamancyc}, \(G\) is infinitesimally rigid if and only if
\(\sum_{i=1}^{s}\nu_t(E_i)\geq 2|V|-1\) holds for \(t=0,1,k-1\), and otherwise \(\sum_{i=1}^{s}\nu_t(E_i)\geq2|V|\) holds for every \(0\leq t\leq k-1\).

Note that $\sum_{i=1}^{s}\nu_t(E_i)\geq\sum_{i=1}^{s}\rho(E_i).$
Thus, by (the end of the proof of) Theorem \ref{thm:forcedgainmain} we get that
\(\sum_{i=1}^{s}\nu_t(E_i)\geq 2|V|\), from which the result follows.
\end{proof}

\section{Further work}

In this paper we gave sufficient connectivity conditions for the two different types of rigidity of symmetric frameworks in the plane for every point group for which a combinatorial characterisation of rigidity is known. The main open question remaining is whether one can give similar conditions for the remaining point groups in the plane or for higher dimensional symmetric frameworks. As finding combinatorial characterisations for rigidity in those cases are key open problems of the field, proving a similar result would require a method different from the one used in this paper.

A similar question can be asked about (infinite) periodic frameworks. As combinatorial characterisations for rigidity of periodic frameworks are known in many cases, this problem seems to be easier to attack than the ones mentioned above. For periodic frameworks 
it is probably 
most suitable to try to give a connectivity condition for the quotient gain graph instead of one for the covering graph. 

Another problem that arises from the investigation of \(n\)-gain-mixed-connected group-labelled graphs is whether this property can be checked in polynomial time. When \(n\) is fixed, then there is a trivial polynomial time algorithm that can decide \(n\)-gain-mixed-connectivity by checking whether \(G\) has a \(k\)-block for all \(k\leq n-1\). As \(|v_H|\geq1\) for every \(v\in V\) and \(|e_H|\geq1\) for every \(e\in E\) it suffices to consider every pair \((U,D)\), \(U\subseteq V\), \(D\subseteq E\) with \(2|U|+|D|\leq k\) and check whether one of the components of \(G-U-D\) is a \(k\)-block. This can be done by making the labels of a spanning tree in a component the identity. The group induced by the remaining edges is the symmetry group of the component. In this paper we only consider the cases \(n=6\) and 7 so this algorithm can be used for checking whether the sufficient conditions in our results hold for \(G\). The complexity of this algorithm depends on \(n\). Thus designing an efficient algorithm that finds the largest \(n\) for which \(G\) is \(n\)-gain-mixed-connected remains an open problem.

\section{Acknowledgements}

  The first author was supported by the Hungarian Scientific Research Fund of the National Research, Development and Innovation Office (OTKA, grant number K109240 and K124171). The second author was supported by EPSRC First Grant EP/M013642/1.

\end{document}